\documentclass[10pt,reqno]{amsart}
\usepackage{amsmath, amsthm, amscd, amsfonts, amssymb, graphicx, color}
\newtheorem{theorem}{Theorem}[section]
\newtheorem{lemma}[theorem]{Lemma}

\newtheorem{corollary}[theorem]{Corollary}
\theoremstyle{definition}
\newtheorem{definition}[theorem]{Definition}
\newtheorem{example}[theorem]{Example}

\theoremstyle{remark}
\newtheorem{remark}[theorem]{Remark}

\begin{document}

\title[Hermite-Hadamard's Mid-Point Type Inequalities for Generalized...]{Hermite-Hadamard's Mid-Point Type Inequalities for Generalized Fractional Integrals}

\author[M. Rostamian Delavar]{M. Rostamian Delavar}
\address{Department of Mathematics, Faculty of Basic Sciences, University of Bojnord, P. O. Box 1339, Bojnord 94531, Iran}
\email{\textcolor[rgb]{0.00,0.00,0.84}{m.rostamian@ub.ac.ir}}




\begin{abstract}
Some Hermite-Hadamard's mid-point type inequalities related to Katugampola fractional integrals are obtained where the first derivative of considered mappings is Lipschitzian or convex. Also some mid-point type inequalities are given for Lipschitzian mappings, with the aim of generalizing the results presented in previous works. Finally as an application, some generalized inequalities in connection with special means are provided.
\end{abstract}

\maketitle

\section{Introduction}
Recently U. N. Katugampola in \cite{katu1}, introduced an Erd\'elyi-Kober type fractional integral operator which now is known as Katugampola fractional integral. The Katugampola fractional integral is a generalization of Riemann-Liouville and Hadamard fractional integrals simultaneously. Let's review these concepts.\\
\noindent The following definition is modified version of Definition 4.3 in \cite{katu1}.
\begin{definition}\cite{katu3, katu2}
Let $[a,b]\subset \mathbb{R}$ be a finite interval. The left and right side Katugampola fractional integrals of order $\alpha>0$ are defined respectively by
\begin{align*}
{^\rho I_{a^{+}}^{\alpha}} f(x)=\frac{\rho^{1-\alpha}}{\Gamma({\alpha})}\int_a^x\frac{t^{\rho-1}}{(x^{\rho}-t^{\rho})^{1-\alpha}}f(t)dt,
\end{align*}
and
\begin{align*}
{^\rho I_{b^{-}}^{\alpha}} f(x)=\frac{\rho^{1-\alpha}}{\Gamma({\alpha})}\int_x^b\frac{t^{\rho-1}}{(t^{\rho}-x^{\rho})^{1-\alpha}}f(t)dt,
\end{align*}
where $a<x<b$, $\rho>0$, $\Gamma(\alpha)$ is Gamma function and the integrals exist.
\end{definition}
\noindent The relation between Riemann-Liouville fractional integrals i.e.,
\begin{align*}
 J_{a^{+}}^{\alpha} f(x)=\frac{1}{\Gamma(\alpha)}\int_a^x (x-t)^{\alpha-1}f(t)dt~~\quad \text{~and~} ~~\quad J_{b^{-}}^{\alpha} f(x)=
 \frac{1}{\Gamma({\alpha})}\int_x^b(t-x)^{\alpha-1}f(t)dt,
\end{align*}
and Hadamard fractional integrals i.e.,
\begin{align*}
 H_{a^{+}}^{\alpha} f(x)=\frac{1}{\Gamma(\alpha)}\int_a^x \Big(\ln \frac{x}{t}\Big)^{\alpha-1}f(t)dt~~\quad \text{~and~} ~~\quad H_{b^{-}}^{\alpha} f(x)=
 \frac{1}{\Gamma({\alpha})}\int_x^b\Big(\ln \frac{x}{t}\Big)^{\alpha-1}f(t)dt,
\end{align*}
has been shown in the following result:
\begin{theorem}\cite{katu3,katu2}
Let $\alpha>0$ and $\rho>0$. Then for $x>a$\\

\rm{(a)} $\lim_{\rho\to 1} {^\rho I_{a^{+}}^{\alpha}} f(x)= J_{a^{+}}^{\alpha} f(x)$ and $\lim_{\rho\to 1} {^\rho I_{b^{-}}^{\alpha}} f(x)= J_{b^{-}}^{\alpha} f(x)$,\\

\rm{(b)} $\lim_{\rho\to 0^{+}} {^\rho I_{a^{+}}^{\alpha}} f(x)= H_{a^{+}}^{\alpha} f(x)$ and $\lim_{\rho\to 0^{+}} {^\rho I_{b^{-}}^{\alpha}} f(x)= H_{b^{-}}^{\alpha} f(x)$.
\end{theorem}
\noindent ّFor basic and fundamental information about fractional integrals and operators we refer an interested reader to \cite{goma, kilbas, kirya, samko}.\\
In \cite{katu3}, the authors obtained two important inequalities in connection with Hermite-Hadamard inequality and Katugampola fractional integrals. The first is Hermite-Hadamard type inequality related to Katugampola fractional integrals:
\begin{theorem}
Let $\alpha>0$ and $\rho>0$. Let $f:[a^{\rho},b^{\rho}]\to \mathbb{R}$ be a positive function with $0\leq a<b$ and $f\in X_c^p(a^{\rho},b^{\rho})$. If $f$ is also a convex function on $[a,b]$,
then the following inequalities hold:
\begin{align}\label{eq.-10}
&f\Big(\dfrac{a^{\rho}+b^{\rho}}{2}\Big)\leq\frac{\alpha{\rho}^{\alpha}\Gamma(\alpha +1)}{2(b^{\rho}-a^{\rho})^{\alpha}}\big[{^\rho I_{a^{+}}^{\alpha}} f(b^\rho )+{^\rho I_{b^{-}}^{\alpha}} f(a^\rho )\big]\leq \dfrac{f(a^{\rho})+f(b^{\rho})}{2}.
\end{align}
\end{theorem}
\noindent Note that inequalities obtained in (\ref{eq.-10}), generalize the Hermite-Hadamard inequality related to Riemann-Lioville fractional integrals presented by M. Z. Sarikaya et al. \cite{sarikaya}(also see \cite{wang}):
\begin{align}\label{eq.--10}
&f\Big(\dfrac{a+b}{2}\Big)\leq\frac{\Gamma(\alpha +1)}{2(b-a)^{\alpha}}\big[ J_{a^{+}}^{\alpha} f(b )+J_{b^{-}}^{\alpha} f(a )\big]\leq \dfrac{f(a)+f(b)}{2}.
\end{align}
If in (\ref{eq.--10}) we consider $\alpha=1$, then we recapture classic Hermite-Hadamard inequality \cite{hadamard, hermite, mitri} for a convex function $f$ on $[a,b]$:
\begin{align*}
&f\Big(\dfrac{a+b}{2}\Big)\leq\frac{1}{b-a}\int_a^b f(x)dx\leq \dfrac{f(a)+f(b)}{2}.
\end{align*}
For more results about Hermite-Hadamard inequality and fractional integrals see \cite{chen, iscan, jleli, noor, rode, RD, set, wang} and references therein.\\
\noindent The second is the following inequality in connection with (\ref{eq.-10}):
\begin{theorem}\label{thm.000}
 Let $f:[a^{\rho},b^{\rho}]\to \mathbb{R}$ be a differentiable mapping on $(a^{\rho},b^{\rho})$ with $0\leq a<b$. If $|f'|$ is convex on $[a^{\rho},b^{\rho}]$, then the following inequality holds: \begin{align}\label{eq.-9}
&\bigg|\dfrac{f(a^{\rho})+f(b^{\rho})}{2}-\frac{\alpha{\rho}^{\alpha}\Gamma(\alpha +1)}{2(b^{\rho}-a^{\rho})^{\alpha}}\big[{^\rho I_{a^{+}}^{\alpha}} f(b^\rho )+{^\rho I_{b^{-}}^{\alpha}} f(a^\rho )\big]\bigg|\leq\dfrac{b^{\rho}-a^{\rho}}{2\rho(\alpha+1)}\big(1-\frac{1}{2^{\alpha}}\big)[|f'(a^{\rho})|+|f'(b^{\rho})|].
\end{align}
\end{theorem}
\noindent We call (\ref{eq.-9}) as trapezoid type inequality in connection with (\ref{eq.-10}), because of the geometric interpretation contained in the following interesting classic inequality obtained by S. S. Dragomir et al. in \cite{dragomir}.
\begin{theorem}Let $f:I^\circ\subseteq \mathbb{R}\to \mathbb{R}$ be a differentiable mapping on $I^{\circ}$, $a,b\in I^\circ$ with $a<b$.
If $|f'|$ is convex on $[a,b]$, then the following inequality holds:
\begin{align*}
\Big{|}\frac{f(a)+f(b)}{2}(b-a)-\int_a^bf(x)dx\Big{|}\leq \frac{{(b-a)}^2}{8}\Big{(}|f'(a)|+|f'(b)|\Big{)}.
\end{align*}
\end{theorem}
\noindent Also U. S. Kirmaci, in \cite{kirmaci} obtained another classic inequality related to Hermite-Hadamard inequality as the following:
\begin{theorem}\label{kir}
 Consider $I^*$ as the interior of interval $I\subset\mathbb{R}$. Let $f:I^*\to\mathbb{R}$ be a differentiable mapping on $I^*$, $a, b\in I^*$ with $a < b$. If $|f'|$ is convex on $[a, b]$, then we have
\begin{align}\label{eq.-7}
\Big{|}\int_a^bf(x)dx-(b-a)f\Big{(}\frac{a+b}{2}\Big{)}\Big{|}\leq \frac{{(b-a)}^2}{8}\Big{(}|f'(a)|+|f'(b)|\Big{)}.
\end{align}
\end{theorem}
\noindent Because of the geometric interpretation contained in above result, we call (\ref{eq.-7}) as mid-point type inequality in connection with Hermite-Hadamard inequality.

In this paper, Motivated by above works, we obtain some mid-point type inequalities related to Katugampola fractional integrals. We consider Lipschitzian mappings and also the functions whose the first derivative is Lipschitzian. Furthermore we obtain some results for functions whose first derivative absolute values are convex. As an application, some generalized inequalities in connection with two important special means are provided. Some examples and corollaries support our results.
\section{Mid-Point Type Inequalities}
\noindent In this section we obtain three Hermite-Hadamard's mid-point type theorems related to Katugampola fractional integrals by considering the concepts of Lipschitzian and convex mappings. The following lemma
is of importance to achieve our main results.
\begin{lemma}\label{lemma01}
Let $f:I\to \mathbb{R}$ be a differentiable function on $I^{\circ}$. For $0\leq a<b$ and $\rho>0$, suppose that $a^{\rho},b^{\rho}\in I^\circ$ and $f'\in L[a^{\rho},b^{\rho}]$. Then for $\alpha>0$, the following identities for fractional integrals hold:
\begin{align}\label{eq.03}
&-\rho (b^{\rho}-a^{\rho})\bigg\{\int_0^{\frac{1}{\sqrt[\rho]{2}}} t^{(\alpha+1)\rho-1} f'\big(t^{\rho}a^{\rho}+(1-t^{\rho})b^{\rho}\big)dt+\\
& \int_{\frac{1}{\sqrt[\rho]{2}}}^1 (t^{\alpha\rho}-1)t^{\rho-1} f'\big(t^{\rho}a^{\rho}+(1-t^{\rho})b^{\rho}\big)dt\bigg\}=\notag\\
&f\Big(\dfrac{a^{\rho}+b^{\rho}}{2}\Big)-\frac{\alpha{\rho}^{\alpha}\Gamma(\alpha +1)}{(b^{\rho}-a^{\rho})^{\alpha}}{^\rho I_{a^{+}}^{\alpha}} f(b^\rho ),\notag
\end{align}
and
\begin{align}\label{eq.04}
&\rho (b^{\rho}-a^{\rho})\bigg\{\int_0^{\frac{1}{\sqrt[\rho]{2}}} t^{(\alpha+1)\rho-1} f'\big(t^{\rho}b^{\rho}+(1-t^{\rho})a^{\rho}\big)dt+\\
& \int_{\frac{1}{\sqrt[\rho]{2}}}^1 (t^{\alpha\rho}-1)t^{\rho-1} f'\big(t^{\rho}b^{\rho}+(1-t^{\rho})a^{\rho}\big)dt\bigg\}=\notag\\
&f\Big(\dfrac{a^{\rho}+b^{\rho}}{2}\Big)-\frac{\alpha{\rho}^{\alpha}\Gamma(\alpha +1)}{(b^{\rho}-a^{\rho})^{\alpha}}{^\rho I_{b^{-}}^{\alpha}} f(a^\rho ).\notag
\end{align}
Furthermore
\begin{align}\label{eq.05}
&\frac{\rho (b^{\rho}-a^{\rho})}{2}\bigg\{\int_0^{\frac{1}{\sqrt[\rho]{2}}} t^{(\alpha+1)\rho-1} \Big[f'\big(t^{\rho}b^{\rho}+(1-t^{\rho})a^{\rho}\big)-f'\big(t^{\rho}a^{\rho}+(1-t^{\rho})b^{\rho}\big)\Big]dt+\\
& \int_{\frac{1}{\sqrt[\rho]{2}}}^1 (t^{\alpha\rho}-1)t^{\rho-1}\Big[f'\big(t^{\rho}b^{\rho}+(1-t^{\rho})a^{\rho}\big)- f'\big(t^{\rho}a^{\rho}+(1-t^{\rho})b^{\rho}\big)\Big]dt\bigg\}=\notag\\
&f\Big(\dfrac{a^{\rho}+b^{\rho}}{2}\Big)-\frac{\alpha{\rho}^{\alpha}\Gamma(\alpha +1)}{2(b^{\rho}-a^{\rho})^{\alpha}}\big[{^\rho I_{a^{+}}^{\alpha}} f(b^\rho )+{^\rho I_{b^{-}}^{\alpha}} f(a^\rho )\big].\notag
\end{align}
\end{lemma}

\begin{proof}
By the use of integration by parts we get
\begin{align}\label{eq.07}
&\int_0^{\frac{1}{\sqrt[\rho]{2}}} t^{\alpha\rho}t^{\rho-1} f'\big(t^{\rho}a^{\rho}+(1-t^{\rho})b^{\rho}\big)dt= t^{\alpha\rho}\frac{f\big(t^{\rho}a^{\rho}+(1-t^{\rho})b^{\rho}\big)}{\rho(a^{\rho}-b^{\rho})}\bigg|_0^{\frac{1}{\sqrt[\rho]{2}}}-\\
&\int_0^{\frac{1}{\sqrt[\rho]{2}}}\frac{\alpha t^{\alpha\rho-1}f\big(t^{\rho}a^{\rho}+(1-t^{\rho})b^{\rho}\big)}{a^{\rho}-b^{\rho}}dt=\frac{-(\frac{1}{2})^{\alpha}}{\rho(b^{\rho}-
a^{\rho})}f\Big(\frac{a^{\rho}+b^{\rho}}{2}\Big)+\notag\\
&\frac{\alpha}{b^{\rho}-a^{\rho}}\int_0^{\frac{1}{\sqrt[\rho]{2}}} t^{\alpha\rho-1}f\big(t^{\rho}a^{\rho}+(1-t^{\rho})b^{\rho}\big)dt.\notag
\end{align}
Similarly we have
\begin{align}\label{eq.08}
&\int_{\frac{1}{\sqrt[\rho]{2}}}^1 (t^{\alpha\rho}-1)t^{\rho-1}f'\big(t^{\rho}a^{\rho}+(1-t^{\rho})b^{\rho}\big)dt=\frac{(\frac{1}{2})^{\alpha}-1}{\rho(b^{\rho}-a^{\rho})}f\Big(\frac{a^{\rho}+b^{\rho}}{2}\Big)+\\
&\frac{\alpha}{b^{\rho}-a^{\rho}}\int_{\frac{1}{\sqrt[\rho]{2}}}^{1} t^{\alpha\rho-1}f\big(t^{\rho}a^{\rho}+(1-t^{\rho})b^{\rho}\big)dt.\notag
\end{align}
Now merging (\ref{eq.07}) and (\ref{eq.08}) with applying the change of variable $x^{\rho}=t^{\rho}a^{\rho}+(1-t^{\rho})b^{\rho}$ imply that
\begin{align}
&\int_0^{\frac{1}{\sqrt[\rho]{2}}} t^{\alpha\rho}t^{\rho-1} f'\big(t^{\rho}a^{\rho}+(1-t^{\rho})b^{\rho}\big)dt+\int_{\frac{1}{\sqrt[\rho]{2}}}^1 (t^{\alpha\rho}-1)t^{\rho-1}f'\big(t^{\rho}a^{\rho}+(1-t^{\rho})b^{\rho}\big)dt=\notag\\
&\frac{\alpha}{(b^{\rho}-a^{\rho})}\int_0^{1} t^{\alpha\rho-1}f\big(t^{\rho}a^{\rho}+(1-t^{\rho})b^{\rho}\big)dt-\frac{1}{\rho(b^{\rho}-
a^{\rho})}f\Big(\frac{a^{\rho}+b^{\rho}}{2}\Big)=\notag\\
&\frac{\alpha}{(b^{\rho}-a^{\rho})}\int_a^{b}\Big(\frac{b^{\rho}-x^{\rho}}{b^{\rho}-a^{\rho}}\Big)^{\alpha-1}f(x^{\rho})\frac{x^{\rho-1}}{b^{\rho}-a^{\rho}}dx-\frac{1}{\rho(b^{\rho}-
a^{\rho})}f\Big(\frac{a^{\rho}+b^{\rho}}{2}\Big)=\notag\\
&\frac{\alpha{\rho}^{\alpha-1}\Gamma(\alpha+1)}{(b^{\rho}-a^{\rho})^{\alpha+1}}{^\rho I_{a^{+}}^{\alpha}f(b^{\rho})}-\frac{1}{\rho(b^{\rho}-
a^{\rho})}f\Big(\frac{a^{\rho}+b^{\rho}}{2}\Big).\notag
\end{align}
Note that for identity (\ref{eq.04}), the proof is similar. To prove (\ref{eq.05}), it is enough to add identity (\ref{eq.03}) to (\ref{eq.04}).
\end{proof}
\subsection{$f'$ and $f$ are Lipschitzian Mappings}
\begin{definition}\cite{robert}
A function $f : [a, b]\to\mathbb{R}$ is said to satisfy a Lipschitz condition on interval
$[a, b]$ ($M$-Lipschitzian) if there is a constant M so that, for any two points $x, y\in [a, b]$,

$$|f (x)-f (y)|\leq M|x-y|.$$
\end{definition}

\noindent By the use of Lemma \ref{lemma01}, we can obtain a new mid-point type theorem in the case that first derivative of considered function is Lipschitzian.
\begin{theorem}\label{thm.02}
Let $f:I\to \mathbb{R}$ be a differentiable function on $I^{\circ}$. For $0\leq a<b$ and $\rho>0$, suppose that $a^{\rho},b^{\rho}\in I^{\circ}$ and $f'$ satisfies a Lipschitz condition on $[a^{\rho},b^{\rho}]$ with respect to $M$. Then for $\alpha>0$, the following mid-point type inequality holds:
\begin{align}\label{eq.10}
& \bigg|f \Big(\dfrac{a^{\rho}+b^{\rho}}{2}\Big)- \dfrac{\alpha{\rho}^{\alpha}\Gamma(\alpha +1)}{2(b^{\rho}-a^{\rho})^{\alpha}} \Big[{^\rho I_{a^{+}}^{\alpha}} f\big(b^{\rho})+{^\rho I_{b^{-}}^{\alpha}} f(a^\rho )\Big ] \bigg|\leq\dfrac{M(b^{\rho}-a^{\rho})^2(\alpha^2-\alpha+2)}{8{(\alpha+1})(\alpha+2)}.
\end{align}
\end{theorem}
\begin{proof}
From identity (\ref{eq.05}) we have
\begin{align}
&\frac{\rho(b^{\rho}-a^{\rho})}{2}\bigg\{\int_0^{\frac{1}{\sqrt[\rho]{2}}} t^{(\alpha+1)\rho-1}\Big|f'\big(t^{\rho}b^{\rho}+(1-t^{\rho})a^{\rho}\big)-f'\big(t^{\rho}a^{\rho}+(1-t^{\rho})b^{\rho}\big)\Big|dt +\notag\\
&\int_{\frac{1}{\sqrt[\rho]{2}}}^1 (t^{\alpha\rho}-1)\Big|f'\big(t^{\rho}b^{\rho}+(1-t^{\rho})a^{\rho}\big)-f'\big(t^{\rho}a^{\rho}+(1-t^{\rho})b^{\rho}\big)\Big|dt\bigg\}\leq\notag\\
&\frac{M\rho(b^{\rho}-a^{\rho})}{2}\bigg\{\int_0^{\frac{1}{\sqrt[\rho]{2}}} t^{(\alpha+1)\rho-1}\Big|(2t^{\rho}-1)(b^{\rho}-a^{\rho})\Big|dt +
\int_{\frac{1}{\sqrt[\rho]{2}}}^1 \big|t^{\alpha\rho}-1\big|t^{\rho-1}\Big|(2t^{\rho}-1)(b^{\rho}-a^{\rho})\Big|dt\bigg\}=\notag\\
&\frac{M\rho(b^{\rho}-a^{\rho})^2}{2}\bigg\{\int_0^{\frac{1}{\sqrt[\rho]{2}}} t^{(\alpha+1)\rho-1}(1-2t^{\rho})dt +
\int_{\frac{1}{\sqrt[\rho]{2}}}^1 (1-t^{\alpha\rho})t^{\rho-1}(2t^{\rho}-1)dt\bigg\}=\notag\\
&\frac{M\rho(b^{\rho}-a^{\rho})^2}{2}\bigg\{\frac{(\frac{1}{2})^{\alpha+1}}{(\alpha+1)\rho}-\frac{(\frac{1}{2})^{\alpha+1}}{(\alpha+2)\rho}-\frac{2}{(\alpha+2)\rho}+
\frac{(\frac{1}{2})^{\alpha+1}}{(\alpha+2)\rho}+\frac{1-(\frac{1}{2})^{\alpha+1}}{(\alpha+1)\rho}+\frac{1}{4\rho}\bigg\}=\notag\\
&\frac{M\rho(b^{\rho}-a^{\rho})^2(\alpha^2-\alpha+2)}{8(\alpha+1)(\alpha+2)}.\notag
\end{align}
The details are omitted in calculating of above integrals.
\end{proof}
\begin{corollary}\label{cor.01}
Let $f: I \to \mathbb{R}$ be a differentiable function on $I^\circ$ with $a,b\in I^\circ$. If $f'$ satisfies a Lipschitz condition on $[a,b]$ with respect to $M$, then the following mid-point type inequality holds:
\begin{align}\label{eq.-11}
& \bigg|f \Big(\dfrac{a+b}{2}\Big)- \dfrac{\alpha\Gamma(\alpha +1)}{2(b-a)^{\alpha}} \Big[J_{a^{+}}^{\alpha} f(b)+J_{b^{-}}^{\alpha} f(a)\Big ] \bigg|\leq\dfrac{M(b-a)^2(\alpha^2-\alpha+2)}{8{(\alpha+1})(\alpha+2)}.
\end{align}
Also if we consider $\alpha=1$ in (\ref{eq.-11}) we get
\begin{align}\label{eq.--11}
& \bigg|f \Big(\dfrac{a+b}{2}\Big)- \dfrac{1}{b-a}\int_a^b f(x)dx\bigg|\leq\dfrac{M(b-a)^2}{24},
\end{align}
which is new in literature. \\
Furthermore if in Theorem \ref{thm.02}, we consider that $f$ is twice differentiable on $[a^{\rho},b^{\rho}]$, $f'$ is convex on $[a^{\rho}, b^{\rho}]$ and $M=\sup_{t\in [a^{\rho}, b^{\rho}]}|f''(t)|<\infty$, then by using Lagrange's theorem for any $x,y\in (a^{\rho},b^{\rho})$, there exists a $t\in (x,y)$ such that
$$|f'(x)-f'(y)|=|x-y||f''(t)|\leq M|x-y|,$$
which shows that $f'$ satisfy a Lipschitz condition on $[a^{\rho},b^{\rho}]$ and so again we have (\ref{eq.10}).\\
Finally if $f$ is twice differentiable on $[a^{\rho},b^{\rho}]$, the functions $f$ and $f'$ are convex on $[a^{\rho}, b^{\rho}]$ and $M=\sup_{t\in [a^{\rho}, b^{\rho}]}|f''(t)|<\infty$, then from (\ref{eq.-10}) we get
\begin{align*}
&0\leq \dfrac{\alpha{\rho}^{\alpha}\Gamma(\alpha +1)}{2(b^{\rho}-a^{\rho})^{\alpha}} \Big[{^\rho I_{a^{+}}^{\alpha}} f\big(b^{\rho})+{^\rho I_{b^{-}}^{\alpha}} f(a^\rho )\Big ]
 -f \Big(\dfrac{a^{\rho}+b^{\rho}}{2}\Big)\leq\dfrac{M(b^{\rho}-a^{\rho})^2(\alpha^2-\alpha+2)}{8{(\alpha+1})(\alpha+2)}.
\end{align*}
\end{corollary}

\begin{example}\label{ex.01}
Consider $f(x)=\sin{x}$, $x\in [a,b]$ with $0\leq a<b$. From the fact that $|\cos x-\cos y|\leq |x-y|$, we have that $f'(x)=\cos x$ satisfies a Lipschitz condition with respect to $M=1$. Then from Theorem \ref{thm.02} we have
\begin{align*}
& \bigg|\sin \Big(\dfrac{a^{\rho}+b^{\rho}}{2}\Big)- \dfrac{\alpha{\rho}^{\alpha}\Gamma(\alpha +1)}{2(b^{\rho}-a^{\rho})^{\alpha}} \Big[{^\rho I_{a^{+}}^{\alpha}} \sin(b^{\rho})+{^\rho I_{b^{-}}^{\alpha}} \sin(a^\rho )\Big ] \bigg|\leq\dfrac{(b^{\rho}-a^{\rho})^2(\alpha^2-\alpha+2)}{8{(\alpha+1})(\alpha+2)},
\end{align*}
where
\begin{align*}
&^\rho I_{a^{+}}^{\alpha}\sin (b^{\rho})=\frac{\rho^{1-\alpha}}{\Gamma(\alpha)}\int_a^{b^{\rho}}t^{\rho-1}(b^{\rho}-t^{\rho})^{\alpha-1}\sin t dt =\\
&\frac{\rho^{1-\alpha}}{\Gamma(\alpha)}\bigg\{-\frac{1}{\alpha\rho}(b^{\rho}-b^{\rho^2})^{\alpha}\sin (b^{\rho})+\frac{1}{\alpha\rho}(b^{\rho}-a^{\rho})^{\alpha}
\sin a+\dfrac{1}{\rho} J_{a^{+}}^{\alpha+1}\cos (b^{\rho})\bigg\},
\end{align*}
and
\begin{align*}
&^\rho I_{b^{-}}^{\alpha}\sin (a^{\rho})=\frac{\rho^{1-\alpha}}{\Gamma(\alpha)}\bigg\{\frac{1}{\alpha\rho}(b^{\rho}-a^{\rho})^{\alpha}\sin b-\frac{1}{\alpha\rho}(b^{\rho^2}-a^{\rho})^{\alpha}\sin (a^{\rho})-\dfrac{1}{\rho} J_{b^{-}}^{\alpha+1}\cos (a^{\rho})\bigg\}.
\end{align*}
Now by letting $\rho\to 1$, we obtain that
\begin{align}\label{eq.-11''}
& \bigg|\sin \Big(\dfrac{a+b}{2}\Big)- \dfrac{\alpha^2}{2(b-a)^{\alpha}} \Big[\frac{(b-a)^{\alpha}}{\alpha}\sin a+ J_{a^{+}}^{\alpha+1}\cos b+\frac{(b-a)^{\alpha}}{\alpha}\sin b-J_{b^{-}}^{\alpha+1}\cos a\Big ] \bigg|\leq\\
&\dfrac{(b-a)^2(\alpha^2-\alpha+2)}{8{(\alpha+1})(\alpha+2)}\notag.
\end{align}
If in (\ref{eq.-11''}), we set $\alpha=1$ we get
\begin{align}\label{eq.--11''}
& \bigg|\sin \Big(\dfrac{a+b}{2}\Big)-\frac{\sin a+\sin b}{2}-\frac{J_{a^{+}}^2\cos b+J_{b^{-}}^2\cos a}{2(b-a)}\bigg|\leq\dfrac{(b-a)^2}{24}.
\end{align}
It follows that
\begin{align*}
J_{a^{+}}^2\cos b=\int_a^b(b-t)\cos t dt=-(b-a)\sin a+\cos a-\cos b,
\end{align*}
and
\begin{align*}
J_{b^{-}}^2\cos a=\int_a^b(t-a)\cos t dt=(b-a)\sin b+\cos b-\cos a,
\end{align*}
which along with (\ref{eq.--11''}) we deduce that
\begin{align*}
& \bigg|\sin \Big(\dfrac{a+b}{2}\Big)-\frac{\cos a-\cos b}{b-a}\bigg|\leq\dfrac{(b-a)^2}{24}.\\
\end{align*}
\end{example}

\begin{remark}
If in Example \ref{ex.01} we consider $f(x)=a^x (a>0)$, $f(x)=x^n (n\geq 2)$, $f(x)=\ln x$ and $f(x)=-\frac{1}{x}$, then with the fact that $f'$ satisfies a Lipschitz condition with respect to some $M$ we can obtain some new mid-point estimation type inequalities for $f$ with new bounds.
\end{remark}
To prove the following result, we use the structure presented in \cite{dragomir3} where the considered functions are Lipschitzian.
\begin{theorem}\label{thm.03}
Suppose that for $\rho>0$ and $0\leq a<b$, the function $f: [a^{\rho},b^{\rho}] \to \mathbb{R}$ satisfies a Lipschitz condition on $[a^{\rho},b^{\rho}]$ with respect to $M$. Then the following mid-point type inequality holds:
\begin{align}\label{eq.12}
& \bigg|f \Big(\dfrac{a^{\rho}+b^{\rho}}{2}\Big)- \dfrac{\alpha{\rho}^{\alpha}\Gamma(\alpha +1)}{(b^{\rho}-a^{\rho})^{\alpha}} \Big[\Big(\frac{1}{2}\Big)^{\alpha}{^\rho I_{a^{+}}^{\alpha}} f\big(b^{\rho})+\Big(1-(\frac{1}{2})^{\alpha}\Big){^\rho I_{b^{-}}^{\alpha}} f(a^\rho )\Big ] \bigg|\leq\\
&\dfrac{M(b^{\rho}-a^{\rho})\Big((\frac{1}{2})^{\alpha-1}+\alpha-1\Big)}{2(\alpha+1)}.\notag
\end{align}
\end{theorem}
\begin{proof} For any $t\in [0,1]$, we have
\begin{align}\label{eq.13}
& \Big|t^{\alpha\rho}f(a^\rho) +(1-t^{\alpha\rho})f(b^\rho)- f\big(t^\rho a^{\rho}+(1-t^\rho)b^\rho\big)\Big|\leq\\
&t^{\alpha\rho}\Big|f(a^\rho) -f\big(t^\rho a^{\rho}+(1-t^\rho)b^\rho\big)\Big|+\big|1-t^{\alpha\rho}\big|\Big|f(b^\rho)-
 f\big(t^\rho a^{\rho}+(1-t^\rho)b^\rho\big)\Big|\leq\notag\\
&Mt^{\alpha\rho}\Big|(1-t^\rho)(a^{\rho}-b^\rho)\Big|+M\big|1-t^{\alpha\rho}\big|\Big|t^\rho( b^\rho-a^{\rho})\Big|=M|b^\rho-a^{\rho}|\Big[t^{\alpha\rho}+t^\rho-2t^{(\alpha+1)\rho}\Big].\notag
\end{align}
Now if in (\ref{eq.13}) consider $t=\frac{1}{\sqrt[2]{\rho}}$, then we deduce that
\begin{align}\label{eq.14}
& \bigg|(\frac{1}{2})^{\alpha}f(a^\rho) +\Big(1-(\frac{1}{2})^{\alpha}\Big)f(b^\rho)- f\Big(\frac{a^{\rho}+b^\rho}{2}\Big)\bigg|\leq\\
&M|b^\rho-a^{\rho}|\Big[(\frac{1}{2})^{\alpha}+\frac{1}{2}-2(\frac{1}{2})^{\alpha+1}\Big]=\frac{M}{2}|b^\rho-a^{\rho}|.\notag
\end{align}
If in (\ref{eq.14}) we replace $"a^\rho"$ with $"t^\rho a^{\rho}+(1-t^\rho)b^\rho"$ and replace $"b^\rho"$ with $"t^\rho b^{\rho}+(1-t^\rho)a^\rho"$, then we obtain
\begin{align*}
& \bigg|(\frac{1}{2})^{\alpha}f\big(t^\rho a^{\rho}+(1-t^\rho)b^\rho\big) +\Big(1-(\frac{1}{2})^{\alpha}\Big)f\big(t^\rho b^{\rho}+(1-t^\rho)a^\rho\big)- \\
&f\Big(\frac{t^\rho a^{\rho}+(1-t^\rho)b^\rho+t^\rho b^{\rho}+(1-t^\rho)a^\rho}{2}\Big)\bigg|\leq\frac{M}{2}\big|t^\rho b^{\rho}+(1-t^\rho)a^\rho-t^\rho a^{\rho}-(1-t^\rho) b^{\rho}\big|=\notag\\
&\frac{M(b^{\rho}-a^{\rho})}{2}\big|2t^\rho-1\big|\notag.
\end{align*}
Multiplying above inequality with $"t^{\alpha\rho-1}"$ and then integrating with respect to $t$ on $[0,1]$ imply that
\begin{align*}
& \bigg|(\frac{1}{2})^{\alpha}\int_0^1t^{\alpha\rho-1}f\big(t^\rho a^{\rho}+(1-t^\rho)b^\rho\big) dt+\Big(1-(\frac{1}{2})^{\alpha}\Big)\int_0^1t^{\alpha\rho-1}f\big(t^\rho b^{\rho}+(1-t^\rho)a^\rho\big)dt- \\
&\int_0^1t^{\alpha\rho-1}f\Big(\frac{a^{\rho}+b^\rho}{2}\Big)\bigg|\leq\frac{M(b^{\rho}-a^{\rho})}{2}\int_0^1t^{\alpha\rho-1}\big|2t^\rho-1\big|dt.\notag
\end{align*}
So it follows that
\begin{align}\label{eq.17}
& \bigg|\big(\frac{1}{2}\big)^{\alpha}\frac{\rho^{\alpha-1}\Gamma(\alpha+1)}{(b^{\rho}-a^{\rho})}{^\rho I_{a^{+}}^{\alpha}}+\Big(1-\big(\frac{1}{2}\big)^{\alpha}\Big)\frac{\rho^{\alpha-1}\Gamma(\alpha+1)}{(b^{\rho}-a^{\rho})}{^\rho I_{b^{-}}^{\alpha}}-\frac{1}{\alpha\rho}f\Big(\frac{a^{\rho}+b^\rho}{2}\Big)\bigg|\leq\\
&\frac{M(b^{\rho}-a^{\rho})}{2}\int_0^1t^{\alpha\rho-1}\big|2t^\rho-1\big|dt=\frac{M(b^{\rho}-a^{\rho})}{2}
\bigg[\int_0^{\frac{1}{\sqrt[2]{\rho}}}t^{\alpha\rho-1}dt-2\int_0^{\frac{1}{\sqrt[2]{\rho}}}t^{(\alpha+1)\rho-1}dt+\notag\\
&2\int_{\frac{1}{\sqrt[2]{\rho}}}^1t^{\alpha\rho-1}dt-\int_{\frac{1}{\sqrt[2]{\rho}}}^1t^{\alpha\rho-1}dt\bigg]=
\frac{M(b^{\rho}-a^{\rho})\big[(\frac{1}{2})^{\alpha-1}+\alpha-1\big]}{2\alpha(\alpha+1)\rho}.\notag
\end{align}
Finally by multiplying (\ref{eq.17}) with $"\alpha\rho"$ we obtain (\ref{eq.12}). This completes the proof.
\end{proof}

\begin{corollary} Similar to Corollary \ref{cor.01}, we have that
\begin{align}\label{ineq000}
& \bigg|f \Big(\dfrac{a+b}{2}\Big)- \dfrac{\alpha\Gamma(\alpha +1)}{(b-a)^{\alpha}} \Big[\Big(\frac{1}{2}\Big)^{\alpha}{J_{a^{+}}^{\alpha}} f(b)+\Big(1-(\frac{1}{2})^{\alpha}\Big){J_{b^{-}}^{\alpha}} f(a)\Big ] \bigg|\leq\\
&\dfrac{M(b-a)\Big((\frac{1}{2})^{\alpha-1}+\alpha-1\Big)}{2(\alpha+1)},\notag
\end{align}
and
\begin{align}\label{ineq111}
& \bigg|f \Big(\dfrac{a+b}{2}\Big)- \dfrac{1}{b-a}\int_a^b f(x)dx\bigg|\leq\dfrac{M(b-a)}{4}.
\end{align}
Inequality (\ref{ineq111}) originally obtained in \cite{dragomir3}. Also we can get (\ref{eq.12}) without using absolute value symbol if $f$ is differentiable, convex on $[a^{\rho},b^{\rho}]$ and $M=\sup_{t\in [a^{\rho}, b^{\rho}]}|f'(t)|<\infty$.
\end{corollary}

\begin{example}
In (\ref{ineq000}), consider $f(x)=\tan x$, $x\in [\frac{-\pi}{3},\frac{\pi}{3}]$. For any $a,b\in [\frac{-\pi}{3},\frac{\pi}{3}]$, there exists $t\in (a,b)$ such that
$$1+\tan^2 t=\frac{\tan b-\tan a}{b-a},$$
showing that
$$|\tan b-\tan a|\leq 4|b-a|.$$
So for $0\leq a<b\leq \frac{\pi}{3}$ we have
\begin{align}\label{ineq222}
& \bigg|\tan \Big(\dfrac{a+b}{2}\Big)- \dfrac{\alpha\Gamma(\alpha +1)}{(b-a)^{\alpha}} \Big[\Big(\frac{1}{2}\Big)^{\alpha}{J_{a^{+}}^{\alpha}} \tan (b)+\Big(1-(\frac{1}{2})^{\alpha}\Big){J_{b^{-}}^{\alpha}} \tan (a)\Big ] \bigg|\leq\\
&\dfrac{4(b-a)\Big((\frac{1}{2})^{\alpha-1}+\alpha-1\Big)}{2(\alpha+1)},\notag
\end{align}
where
\begin{align*}
 J_{a^{+}}^{\alpha} \tan(b)=\frac{1}{\Gamma(\alpha)}\int_a^b (x-t)^{\alpha-1}\tan(t)dt,
 \end{align*}
 and
\begin{align*}
J_{b^{-}}^{\alpha} \tan(a)= \frac{1}{\Gamma({\alpha})}\int_a^b(t-x)^{\alpha-1}\tan(t)dt.
\end{align*}
Now if in (\ref{ineq222}) we set $\alpha=1$, then we get
 \begin{align*}
 \bigg|\tan \Big(\dfrac{a+b}{2}\Big)- \dfrac{1}{b-a} \ln \frac{\sec b}{\sec a}\bigg|\leq b-a.
\end{align*}
\end{example}

\begin{remark}
\rm{(1)} For functions $f(x)=\frac{1}{x}$, $f(x)=e^x$ and $f(x)=-\ln x$, we can obtain some inequalities which generalize the corresponding inequalities obtained in Corollary 2.3 in \cite{dragomir3}.

\rm{(2)} Suppose that $f'$ is a Lipschitzian mapping with respect to $M_1$ and $f$ is a Lipschitzian mapping with respect to $M_2$. Comparing two inequalities (\ref{eq.--11}) and (\ref{ineq111}) implies that in the case $M_1<\frac{6M_2}{b-a}$, we have better estimation for mid-point type inequalities.

\rm{(3)} If $f'$ is an $M$-Lipschitzian mapping, then from inequality $\big||f'(x)|-|f'(y)|\big|\leq|f'(x)-f'(y)|$ we have $|f'|$ is Lipschitzian with respect to $M$. So in this case, we can replace $f$ in (\ref{eq.12}) with $|f'|$.
\end{remark}

\subsection{$|f'|$ is Convex}
Now we obtain Hermite-Hadamard's mid-point type inequality related to Katugampola fractional integrals for functions whose the absolute values of first derivative are convex. The Hermite-Hadamard's trapezoid type inequality of this kind is presented in Theorem \ref{thm.000}.
\begin{theorem}\label{thm.01}
Let $f:I\to \mathbb{R}$ be a differentiable function on $I^{\circ}$. For $0\leq a<b$ and $\rho>0$, suppose that $|f'|$ is convex and integrable on $[a^{\rho},b^{\rho}]$. Then in the case that $0<\alpha\rho\leq 1$, the following mid-point type inequality holds:
\begin{align*}
& \bigg|f \Big(\dfrac{a^{\rho}+b^{\rho}}{2}\Big)- \dfrac{\alpha{\rho}^{\alpha}\Gamma(\alpha +1)}{2(b^{\rho}-a^{\rho})^{\alpha}} \Big[{^\rho I_{a^{+}}^{\alpha}} f\big(b^{\rho})+{^\rho I_{b^{-}}^{\alpha}} f(a^\rho )\Big ] \bigg|\leq\dfrac{b^{\rho}-a^{\rho}}{2^{\alpha+1}(\alpha+1)}\Big(|f'(a^{\rho})|+|f'(b^{\rho})|\Big).
\end{align*}
\end{theorem}

\begin{proof} From (\ref{eq.03}) and (\ref{eq.04}) we have
\begin{align}
& \bigg|2f \Big(\dfrac{a^{\rho}+b^{\rho}}{2}\Big)- \dfrac{\alpha{\rho}^{\alpha}\Gamma(\alpha +1)}{(b^{\rho}-a^{\rho})^{\alpha}} \Big[{^\rho I_{a^{+}}^{\alpha}} f\big(b^{\rho})+{^\rho I_{b^{-}}^{\alpha}} f(a^\rho )\Big ] \bigg|\leq\notag\\
& \bigg|f\Big(\dfrac{a^{\rho}+b^{\rho}}{2}\Big)-\frac{\alpha{\rho}^{\alpha}\Gamma(\alpha +1)}{(b^{\rho}-a^{\rho})^{\alpha}}{^\rho I_{a^{+}}^{\alpha}} f(b^\rho )\bigg|+\bigg|f\Big(\dfrac{a^{\rho}+b^{\rho}}{2}\Big)-\frac{\alpha{\rho}^{\alpha}\Gamma(\alpha +1)}{(b^{\rho}-a^{\rho})^{\alpha}}{^\rho I_{b^{-}}^{\alpha}} f(a^\rho )\bigg|\leq\notag\\
&\rho (b^{\rho}-a^{\rho})\bigg\{\int_0^{\frac{1}{\sqrt[\rho]{2}}} t^{(\alpha+1)\rho-1} \big|f'\big(t^{\rho}b^{\rho}+(1-t^{\rho})a^{\rho}\big)\big|dt+\notag\\
&\int_0^{\frac{1}{\sqrt[\rho]{2}}} t^{(\alpha+1)\rho-1}\big|f'\big(t^{\rho}a^{\rho}+(1-t^{\rho})b^{\rho}\big)\big|dt+\int_{\frac{1}{\sqrt[\rho]{2}}}^1 (t^{\alpha\rho}-1)t^{\rho-1}\big|f'\big(t^{\rho}b^{\rho}+(1-t^{\rho})a^{\rho}\big)\big|dt+\notag\\
& \int_{\frac{1}{\sqrt[\rho]{2}}}^1 (t^{\alpha\rho}-1)t^{\rho-1}\big|f'\big(t^{\rho}a^{\rho}+(1-t^{\rho})b^{\rho}\big)\big|dt\notag\bigg\}\leq\\
& \rho (b^{\rho}-a^{\rho})\bigg\{\int_0^{\frac{1}{\sqrt[\rho]{2}}} t^{(\alpha+1)\rho-1}\big[ t^{\rho}|f'(b^{\rho})|+(1-t^{\rho})|f'(a^{\rho})|\big]dt+\notag\\
& \int_0^{\frac{1}{\sqrt[\rho]{2}}} t^{(\alpha+1)\rho-1}\big[ t^{\rho}|f'(a^{\rho})|+(1-t^{\rho})|f'(b^{\rho})|\big]dt+\notag\\
&\int_{\frac{1}{\sqrt[\rho]{2}}}^1 \big|t^{\alpha\rho}-1\big|t^{\rho-1}\big[ t^{\rho}|f'(b^{\rho})|+(1-t^{\rho})|f'(a^{\rho})|\big]dt+\label{eq..11}\\
&\int_{\frac{1}{\sqrt[\rho]{2}}}^1 \big|t^{\alpha\rho}-1\big|t^{\rho-1}\big[ t^{\rho}|f'(a^{\rho})|+(1-t^{\rho})|f'(b^{\rho})|\big]dt\bigg\}\label{eq..22}=\\
&\rho (b^{\rho}-a^{\rho})\bigg\{\frac{1}{2^{\alpha+2}(\alpha+2)\rho}|f'(b^{\rho})|+\frac{\alpha+3}{2^{\alpha+2}(\alpha+1)(\alpha+2)\rho}|f'(a^{\rho})|+\notag\\
&\frac{1}{2^{\alpha+2}(\alpha+2)\rho}|f'(a^{\rho})|+\frac{\alpha+3}{2^{\alpha+2}(\alpha+1)(\alpha+2)\rho}|f'(b^{\rho})|\notag+\\
&\frac{\alpha+3}{2^{\alpha+2}(\alpha+1)(\alpha+2)\rho}|f'(b^{\rho})|+\frac{1}{2^{\alpha+2}(\alpha+2)\rho}|f'(a^{\rho})|\notag+\\
&\frac{\alpha+3}{2^{\alpha+2}(\alpha+1)(\alpha+2)\rho}|f'(a^{\rho})|+\frac{1}{2^{\alpha+2}(\alpha+2)\rho}|f'(b^{\rho})|\bigg\}=\dfrac{b^{\rho}-
a^{\rho}}{2^{\alpha}(\alpha+1)}\Big(|f'(a^{\rho})|+|f'(b^{\rho})|\Big).\notag
\end{align}
Note that in calculations of integrals (\ref{eq..11}) and (\ref{eq..22}) we used the fact that $|1-t^{\alpha\rho}|\leq|1-t|^{\alpha\rho}$,  where $0<\alpha\leq 1$. Some other details are omitted.
\end{proof}

\begin{remark}
Theorem \ref{thm.01}, is a generalized form of Theorem 2 in \cite{iqbal} (consider $\rho=1$) and so is a generalization for Theorem \ref{kir} (consider $\alpha=\rho=1$).
\end{remark}


\section{Special Means}
In this section as an application of our results we obtain some generalized inequalities related to two well known special means:
\begin{align*}
&A(a,b)=\frac{a+b}{2}\qquad\qquad\qquad\qquad\qquad ~~~~~~~~~~~~~~~~arithmetic ~mean,\\
&L_n(a,b)=\Big[\frac{b^{n+1}-a^{n+1}}{(n+1)(b-a)}\Big]^{\frac{1}{n}}\qquad\qquad~generalized ~log\!-\! mean, ~n\in \mathbb{N}, ~a<b.
\end{align*}
In fact we give some generalized estimation type results for the difference of two means. For more concepts and results about special means, see \cite{DP} and references therein.\\
Consider $f(t)=t^n$ for $t\geq 0$, $n\in \mathbb{N}$. Now $M=nb^{(n-1)\rho}$, and so from Theorem \ref{thm.03} we have \\
\begin{align}\label{ineq.-17}
& \bigg|\Big(\dfrac{a^{\rho}+b^{\rho}}{2}\Big)^n- \dfrac{\alpha{\rho}^{\alpha}\Gamma(\alpha +1)}{(b^{\rho}-a^{\rho})^{\alpha}} \Big[\Big(\frac{1}{2}\Big)^{\alpha}{^\rho I_{a^{+}}^{\alpha}} \big(b^{n\rho}\big)+\Big(1-(\frac{1}{2})^{\alpha}\Big){^\rho I_{b^{-}}^{\alpha}}(a^{n\rho} )\Big ] \bigg|\leq\\
&\dfrac{nb^{(n-1)\rho}(b^{\rho}-a^{\rho})\Big((\frac{1}{2})^{\alpha-1}+\alpha-1\Big)}{2(\alpha+1)}.\notag
\end{align}
By using integration by parts for $n$ times we have
\begin{align}\label{ineq.--17}
&{^\rho I_{a^{+}}^{\alpha}} \big(b^{n\rho}\big)=\frac{\rho^{1-\alpha}}{\Gamma(\alpha)}\bigg[\frac{a^n(b^{\rho}-a^{\rho})^{\alpha}}{\alpha\rho}-\frac{b^{n\rho}(b^{\rho}-b^{\rho^2})^{\alpha}}{\alpha\rho}
+\frac{na^{n-\rho}(b^{\rho}-a^{\rho})^{\alpha+1}}{\alpha(\alpha+1)\rho^2}-\\
&\frac{nb^{(n-\rho)\rho}(b^{\rho}-b^{\rho^2})^{\alpha+1}}{\alpha(\alpha+1)\rho^2}+
\frac{n(n-\rho)a^{n-2\rho}(b^{\rho}-a^{\rho})^{\alpha+2}}{\alpha(\alpha+1)(\alpha+2)\rho^3}-
\frac{n(n-\rho)b^{(n-2\rho)\rho}(b^{\rho}-b^{\rho^2})^{\alpha+2}}{\alpha(\alpha+1)(\alpha+2)\rho^3}+\cdots+\notag\\
&\frac{a^{n-(n-1)\rho}(b^{\rho}-a^{\rho})^{\alpha+n-1}\prod_{i=0}^{n-2}(n-i\rho)}{\rho^n\prod_{i=0}^{n-1}(\alpha+i)}
-\frac{b^{(n-(n-1)\rho)\rho}(b^{\rho}-b^{{\rho}^2})^{\alpha+n-1}\prod_{i=0}^{n-2}(n-i\rho)}{\rho^n\prod_{i=0}^{n-1}(\alpha+i)}+\notag\\
&\frac{\prod_{i=0}^{n-1}(n-i\rho)}{\rho^n\prod_{i=0}^{n-1}(\alpha+i)}\int_a^{b^{\rho}}t^{(n-1)(1-\rho)}(b^{\rho}-t^{\rho})^{\alpha+n-1}dt\bigg].\notag
\end{align}
Also
\begin{align}\label{ineq.---17}
&{^\rho I_{b^{-}}^{\alpha}} \big(a^{n\rho}\big)=\frac{\rho^{1-\alpha}}{\Gamma(\alpha)}\bigg[\frac{b^n(b^{\rho}-a^{\rho})^{\alpha}}{\alpha\rho}-\frac{a^{n\rho}(a^{\rho^2}
-a^{\rho})^{\alpha}}{\alpha\rho}-\frac{nb^{n-\rho}(b^{\rho}-a^{\rho})^{\alpha+1}}{\alpha(\alpha+1)\rho^2}+\\
&\frac{na^{(n-\rho)\rho}(a^{\rho^2}-a^{\rho})^{\alpha+1}}{\alpha(\alpha+1)\rho^2}+
\frac{n(n-\rho)b^{n-2\rho}(b^{\rho}-a^{\rho})^{\alpha+2}}{\alpha(\alpha+1)(\alpha+2)\rho^3}-\frac{n(n-\rho)a^{(n-2\rho)\rho}(a^{\rho^2}
-a^{\rho})^{\alpha+2}}{\alpha(\alpha+1)(\alpha+2)\rho^3}-\cdots+\notag\\
&\frac{(-1)^{n-1}b^{n-(n-1)\rho}(b^{\rho}-a^{\rho})^{\alpha+n-1}\prod_{i=0}^{n-2}(n-i\rho)}{\rho^n\prod_{i=0}^{n-1}(\alpha+i)}-\notag\\
&\frac{(-1)^{n-1}a^{(n-(n-1)\rho)\rho}(a^{\rho^2}-a^{\rho})^{\alpha+n-1}\prod_{i=0}^{n-2}(n-i\rho)}{\rho^n\prod_{i=0}^{n-1}(\alpha+i)}+\notag\\
&\frac{(-1)^n\prod_{i=0}^{n-1}(n-i\rho)}{\rho^n\prod_{i=0}^{n-1}(\alpha+i)}\int_{a^{\rho}}^bt^{(n-1)(1-\rho)}(t^{\rho}-a^{\rho})^{\alpha+n-1}dt\bigg].\notag
\end{align}
Now letting $\rho\to 1$ in (\ref{ineq.--17}) and (\ref{ineq.---17}), along with some calculations, implies that:\\
\begin{align}\label{ineq.17'}
& \Bigg|\Big(\dfrac{a+b}{2}\Big)^n- \dfrac{\alpha^2}{(b-a)^{\alpha}} \bigg[\Big(\frac{1}{2}\Big)^{\alpha}\sum_{m=0}^n\frac{a^{n-m}(b-a)^{\alpha+m}P(n,m)}{\prod_{i=0}^m(\alpha+i)}+\\
&\Big(1-(\frac{1}{2})^{\alpha}\Big)\sum_{m=0}^n\frac{(-1)^mb^{n-m}(b-a)^{\alpha+m}P(n,m)}{\prod_{i=0}^m(\alpha+i)}\bigg ] \Bigg|\leq\dfrac{nb^{(n-1)}(b-a)\Big((\frac{1}{2})^{\alpha-1}+\alpha-1\Big)}{2(\alpha+1)}.\notag
\end{align}
where
\begin{align*}
\sum_{m=0}^n\frac{a^{n-m}(b-a)^{\alpha+m}P(n,m)}{\prod_{i=0}^m(\alpha+i)}=\Gamma(\alpha)J_{a^{+}}^{\alpha} (b^n) =\lim_{\rho\to 1}{^\rho I_{a^{+}}^{\alpha}} \big(b^{n\rho}\big),
\end{align*}
\begin{align*}
\sum_{m=0}^n\frac{(-1)^mb^{n-m}(b-a)^{\alpha+m}P(n,m)}{\prod_{i=0}^m(\alpha+i)}=\Gamma(\alpha) J_{b^{-}}^{\alpha} (a^n)=\lim_{\rho\to 1}{^\rho I_{b^{-}}^{\alpha}} \big(a^{n\rho}\big),
\end{align*}
and
\begin{align*}
P(n,m)=\frac{n!}{(n-m)!},
\end{align*}
which is the number of possible permutations of $k$ objects from a set of $n$.\\
In special case if we consider $\alpha=1$, then it is not hard to see that
\begin{align*}
J_{a^{+}}^{1} f(b^n)+ J_{b^{-}}^{1} f(a^n) =\frac{2(b^{n+1}-a^{n+1})}{n+1}.
\end{align*}
So from inequality (\ref{ineq.17'}) we obtain that
\begin{align}\label{ineq.-18}
& \bigg | A^n(a,b)-L_n^n(a,b)\bigg |\leq\dfrac{nb^{n-1}(b-a)}{4}.
\end{align}
So we conclude that inequalities (\ref{ineq.-17}) and (\ref{ineq.17'}) are generalization of inequality (\ref{ineq.-18}), which has been obtained in \cite{dragomir3}.\\
Also with similar argument as above, from Theorem \ref{thm.01} we have
\begin{align}\label{ineq.19}
& \bigg|\Big(\dfrac{a^{\rho}+b^{\rho}}{2}\Big)^n- \dfrac{\alpha{\rho}^{\alpha}\Gamma(\alpha +1)}{2(b^{\rho}-a^{\rho})^{\alpha}} \Big[{^\rho I_{a^{+}}^{\alpha}} \big(b^{n\rho}\big)+{^\rho I_{b^{-}}^{\alpha}}(a^{n\rho} ) ] \bigg|\leq\dfrac{n(b^{\rho}-a^{\rho})}{2^{\alpha+1}(\alpha+1)}\Big(a^{(n-1)\rho}+b^{(n-1)\rho}\Big),
\end{align}
and if $\rho\to 1$, then
\begin{align}\label{ineq.20}
& \Bigg|\Big(\dfrac{a+b}{2}\Big)^n- \dfrac{\alpha^2}{2(b-a)^{\alpha}}\sum_{m=0}^n\frac{[a^{n-m}+(-1)^mb^{n-m}](b-a)^{\alpha+m}P(n,m)}{\prod_{i=0}^m(\alpha+i)}\Bigg|
\leq\dfrac{n(b-a)(a^{n-1}+b^{n-1})}{2^{\alpha+1}(\alpha+1)}.
\end{align}
Now if in (\ref{ineq.20}) we consider $\alpha=1$, then we recapture inequality (3.1) in \cite{kirmaci}:
\begin{align}\label{ineq.21}
& \bigg | A^n(a,b)-L_n^n(a,b)\bigg |\leq\dfrac{n(b-a)}{4}A(a^{n-1},b^{n-1}),
\end{align}
showing that (\ref{ineq.19}) and (\ref{ineq.20}) generalize (\ref{ineq.21}).


\begin{thebibliography}{99}


\bibitem{chen} F. Chen, \textit{Extensions of the Hermite-Hadamard inequality for harmonically convex functions via fractional integrals},
Appl. Math. Comput. \textbf{268} (2015), 121--128.

\bibitem{katu3} H. Chen and U. N. Katugampola, \textit{Hermite-Hadamard and Hermite-Hadamard-Fej\'er type inequalities for generalized fractional integrals}, J. Math. Anal. Appl. \textbf{446} (2017), 1274--1291.


\bibitem{dragomir} S. S. Dragomir and R. P. Agarwal, \textit{Two inequalities for differentiable mappings and applications to special means
 of real numbers and to trapezoidal formula}, Appl. Math. Lett. \textbf{11} (1998), 91--95.

\bibitem{dragomir3} S. S. Dragomir, Y. J. Cho and S. S. Kim, \textit{Inequalities of Hadamard's type for Lipschitzian mappings and their applications}, J. Math. Anal. Appl. \textbf{245} (2000), 489--501.

\bibitem{DP} S. S. Dragomir and C. E. M. Pearce, \textit{Selected topics on Hermite-Hadamard inequalities and applications}, RGMIA Monographs, Victoria University, 2000.
 (ONLINE: http://ajmaa.org/RGMIA/monographs.php/)



\bibitem{goma} R. Gorenflo, F. Mainardi, \textit{Fractional calculus, integral and differential equations of fractional
order}, Springer Verlag, Wien (1997), 223--276.

\bibitem{hadamard} J. Hadamard, \textit{\'Etude sur les propri\`{e}t\'es des fonctions enti\'eres et en particulier d'une fontion consid\'er\'ee par Riemann}, J. Math. Pures. Appl. \textbf{58} (1893) 171--215.

\bibitem{hermite} C. Hermite, \textit{Sur deux limites d'une int\'egrale d\'efinie}, Mathesis, \textbf{3} (1883), 82--83.

\bibitem{iqbal}  M. Iqbal, M. Iqbal Bhatti and K. Nazeer, \textit{Generalization of inequalities analogous to Hermite-Hadamard inequality via fractional integrals}, Bull. Korean Math. Soc. \textbf{52}(3) (2015), 707--716.

\bibitem{iscan} \.{I}. \.{I}\c{s}can and S. Wu, \textit{Hermite-Hadamard type inequalities for harmonically convex functions via fractional integrals}, Appl. Math. Comput. \textbf{238} (2014), 237--244.




\bibitem{jleli} M. Jleli, D. O'Regan and B. Samet, \textit{On Hermite-Hadamard Type Inequalities via Generalized Fractional Integrals}, Turk. J.
Math. \textbf{40} (2016), 1221--1230.

\bibitem{katu1} U. N. Katugampola, \textit{New approach to a generalized fractional integral}, Appl. Math. Comput. \textbf{218}(3) (2011), 860--865.

\bibitem{katu2} U. N. Katugampola, \textit{New approach to generalized fractional derivatives}, Bull. Math. Anal. Appl. \textbf{6}(4) (2014), 1--15.

\bibitem{kilbas} A. A. Kilbas, H. M. Srivastava and J. J. Trujillo, \textit{Theory and Applications of Fractional Differential Equations}, Elsevier, Amsterdam, Netherlands, 2006.

\bibitem{kirmaci} U. S. Kirmaci, \textit{Inequalities for differentiable mappings and applications to special means
of real numbers and to midpoint formula}, Appl. Math. Comp. \textbf{147}(1) (2004), 137--146.

\bibitem{kirya} V. Kiryakova, \textit{Generalized fractional calculus and applications}, John Wiley \& Sons Inc., New York, 1994.

\bibitem{mitri} D. S. Mitrinovi\'c, I. B. Lackovi\'c, “Hermite and convexity”, Aequationes Math. \textbf{28} (1985) 229--232.

\bibitem{noor} M. A. Noor, K. I. Noor M. U. Awan, and S. Khan, \textit{Fractional Hermite-Hadamard inequalities for some new classes of
Godunova-Levin functions}, Appl. Math. Inf. Sci. \textbf{8}(6) (2014), 2865--2872.






\bibitem{robert} A. W. Robert and D. E. Varbeg, \textit{Convex Functions}, Academic Press, New York and London (1973).

\bibitem{rode} M. Rostamian Delavar and M. De La Sen, \textit{Hermite-Hadamard-Fej\'er Inequality Related to Generalized
Convex Functions via Fractional Integrals}, Journal of Mathematics, \textbf{2018} (2018), Article ID 5864091, 10 pages.

\bibitem{RD} M. Rostamian Delavar and S. S. Dragomir, \textit{Weighted trapezoidal inequalities related to the area balance of a function with applications}, Appl. Math. Comput. \textbf{340} (2019), 5--14.

\bibitem{samko} S. G. Samko, A. A. Kilbas and O. I. Marichev, \textit{Fractional Integrals and Derivatives. Theory and Applications}, Gordon and
Breach, Amsterdam, 1993.


\bibitem{sarikaya}  M. Z. Sarikaya, E. Set, H. Yaldiz, N. Ba\c{s}ak, \textit{Hermite-Hadamard's inequalities for fractional integrals and related fractional inequalities},
Math. Comput. Model. \textbf{57} (2013), 2403--2407.

\bibitem{set} E. Set, \.{I}. \.{I}\c{s}can, M. Z. Sarikaya and M. E. \"Ozdemir, \textit{On new inequalities of Hermite-Hadamard-Fej\'er type for convex functions
via fractional integrals}, Appl. Math. Comput. \textbf{259} (2015), 875--881.

\bibitem{wang} J. Wang and M. Fe\v{c}kan, \textit{Fractional Hermite-Hadamard Inequalities}, De Gruyter, Germany, 2018.

\end{thebibliography}
\end{document}